\documentclass[11pt]{amsart}

\usepackage{amsthm} 
\usepackage{amssymb}
\usepackage{graphicx}									 
\usepackage{stmaryrd}                  
\usepackage{footmisc}                  
\usepackage{paralist}
\usepackage{wrapfig}
\usepackage[draft]{pdfcomment}
\usepackage{bbm}
\usepackage[T2A,T1]{fontenc}
\usepackage[utf8]{inputenc}
\usepackage[ngerman,russian,british]{babel}
\usepackage[arrow, matrix, curve]{xy}
\usepackage{color}

\usepackage{enumitem}

\newtheorem*{thm-plain}{Theorem}

\newtheorem{thm}{Theorem}[section]
\newtheorem*{thmw}{Theorem}
\newtheorem{lem}[thm]{Lemma}

\theoremstyle{definition}

\theoremstyle{remark}

\newcommand{\Co}{\mathbb{C}}
\newcommand{\R}{\mathbb{R}}
\newcommand{\N}{\mathbb{N}}
\newcommand{\Z}{\mathbb{Z}}

\newcommand{\alt}{\mathfrak{A}}
\newcommand{\sym}{\mathfrak{S}}
\newcommand{\dih}{\mathfrak{D}}
\newcommand{\cyc}{\mathfrak{C}}

\newcommand{\SOr}{\mathrm{SO}}

\newcommand{\PSL}{\mathrm{PSL}}

\newcommand{\Or}{\mathrm{O}}

\newcommand{\A}{\mathrm{A}}
\newcommand{\BC}{\mathrm{BC}}
\newcommand{\D}{\mathrm{D}}
\newcommand{\Rr}{\mathrm{R}}
\newcommand{\Qq}{\mathrm{Q}}
\newcommand{\Pp}{\mathrm{P}}
\newcommand{\Ss}{\mathrm{S}}
\newcommand{\Tt}{\mathrm{T}}

\newcommand{\E}{\mathrm{E}}
\newcommand{\F}{\mathrm{F}}

\newcommand{\Hn}{\mathrm{H}}

\newcommand{\To}{\rightarrow}

\newcommand{\MTo}{\mapsto}

\newcommand{\plustimes}{\mathpalette\plustimesinner\relax}
\newcommand{\plustimesinner}[2]{%
  \mathbin{\vphantom{+}\ooalign{$#1+$\cr\hidewidth$#1\times$\hidewidth\cr}}%
}

\newcommand{\Wp}{W^{\scriptscriptstyle+}}

\newcommand{\Wt}{W^{\scriptscriptstyle\times}}

\newcommand{\Mt}{M^{\scriptscriptstyle\times}}
\newcommand{\Gt}{G^{\scriptscriptstyle\times}}
\newcommand{\Gs}{G^{\scriptscriptstyle\plustimes}}

\newcommand{\Ws}{W^{\scriptscriptstyle\plustimes}}

\newcommand{\subgr}{<}

\title[Characterization of reflection-rotation groups]{Characterization of finite groups generated by reflections and rotations}
\author{Christian Lange}
\thanks{The results of this paper appear in the author's thesis written at the University of Cologne \cite{Lange_thesis}. The author was supported by a `Kurzzeitstipendium für Doktoranden' by the German Academic Exchange Service (DAAD)}
\address{Christian Lange, Mathematisches Institut der Universit\"at zu K\"oln, Weyertal 86-90, 50931 K\"oln, Germany}
\email{clange@math.uni-koeln.de}

\begin{document}

\begin{abstract} We characterize finite groups $G$ generated by orthogonal transformations in a finite-dimensional Euclidean space $V$ whose fixed point subspace has codimension one or two in terms of the corresponding quotient space $V/G$.
\end{abstract}\maketitle

\section{Introduction}
\label{sec:Pse_Introduction}


We call a finite group generated by \emph{reflections} and \emph{rotations} in a finite-dimensional Euclidean space, i.e. by orthogonal transformations with codimension \emph{one} and \emph{two} fixed point subspaces, a \emph{reflection-rotation group}. We call it a \emph{rotation group} if it is generated by rotations. For a finite orthogonal group $G$ acting on $\R^n$ the quotient space $\R^n/G$ inherits different structures from $\R^n$, e.g. a topology, a metric and a piecewise linear (PL) structure. It is easy to see that $G$ is a reflection-rotation group, if the quotient space $\R^n/G$ is a piecewise linear manifold with boundary and that in this case $G$ contains a reflection if and only if the boundary of $\R^n/G$ is nonempty (cf. Section \ref{sec:only-if}). In this paper we show that the converse also holds.

\begin{thmw} \label{thm:mikhailovas_theorem} For a finite subgroup $G<\Or_n$ the quotient space $\R^n/G$ is a PL manifold with boundary, if and only if $G$ is a reflection-rotation group. In this case $\R^n/G$ is either PL homeomorphic to the half space $\R^{n-1}\times\R_{\geq 0}$ and $G$ contains a reflection or $\R^n/G$ is PL homeomorphic to $\R^n$ and $G$ does not contain a reflection.
\end{thmw}
A partial result in this direction is contained in a paper by Mikha\^ilova from 1984 \cite{Mikhailova}. In that paper the if direction of our result is verified for many rotation groups but only in the topological category. Large classes of reflection-rotation groups are real reflection groups, their orientation preserving subgroups and unitary reflection groups considered as real groups. However, there are many more examples (cf. Appendix \ref{app:class-rr-groups}), in particular, since a reducible reflection-rotation group in general does not split as a product of irreducible components. A complete classification of reflection-rotation groups has only been obtained recently \cite{LaMik} based on earlier work by Mikha\^ilova \cite{Maerchik,MR608821,Mikhailova2}, Huffman and Wales \cite{MR570727,MR0401936,MR0401937,MR0435243} and others. It comprises both exceptional rotation groups and building blocks of infinite families of rotation groups that are not considered in \cite{Mikhailova}. The largest irreducible rotation group among them occurs in dimension $8$. It is an extension of the alternating group on $8$ letters by a nonabelian group of order $2^7$ and in a sense the largest exceptional irreducible rotation group (cf. \cite{LaMik}). Reducible rotation groups that have not been studied in \cite{Mikhailova} occur, for example, in the product of two copies of a reflection group $W$ of type $\Hn_3$ or $\Hn_4$ due to the existence of outer automorphisms of $W$ that map reflections onto reflections but cannot be realized through conjugation by elements in its normalizer.

We adapt some of the methods from \cite{Mikhailova} as to also work in the piecewise linear category and describe new methods to prove the if direction of our result by verifying its conclusion for all reflection-rotation groups. For instance, we apply a result on equivariant smoothing of piecewise linear manifolds \cite{Lange} and the equivariant Poincar\'e conjecture in the case $n=4$ (cf. Section \ref{sub:dim_four_groups_hom}) and the generalized Poincar\'e conjecture in the case $n>5$ (cf. Section \ref{sub:poincare}). In this way we moreover simplify the consideration compared to \cite{Mikhailova} and make the proofs therein rigorous (cf. Sections \ref{sub:dim_four_groups_hom}, \ref{sub:exceptional_groups_hom}). Finally, in Section \ref{ch:outlook} we suggest an approach to prove our main result that would essentially dispense with the classification of reflection-rotation groups.

Notice that the only-if direction of our result does not hold in the topological category. The quotient $S^3/P$ of a $3$-sphere by the binary icosahedral group $P<\SOr_4$ is Poincar\'e's homology sphere and it follows from Cannon's double suspension theorem that its double suspension $\Sigma^2(S^3/P)$ is a topological $5$-sphere \cite{MR541330}. Therefore, the quotient space $\R^2 \times \R^4/P$ is homeomorphic to $\R^6$, though $P$ is not a rotation group. However, using this paper's result it can be shown that the binary icosahedral group is essentially the only counterexample to the only-if direction in the topological category (cf. \cite{Lange1}).\newline

\emph{Acknowledgements.} I thank Alexander Lytchak for introducing me to the problem and for helpful discussions. I thank the referee for carefully reading the paper and for making valuable comments that helped to improve the exposition.

\section{Notation}
\label{sec:not_ter}

We denote the cyclic group of order $n$ and the dihedral group of order $2n$ by $\cyc_n$ and $\dih_n$, respectively. We denote the symmetric and alternating group on $n$ letters by $\sym_n$ and $\alt_n$. Classical Lie groups are denoted like $\mathrm{SO}_n$ and $\Or_n$. For a finite subgroup $G<\Or_n$ we denote its orientation preserving subgroup as $G^{\scriptstyle+}<\SOr_n$. In particular, we write $\Wp$ for the orientation preserving subgroup of a reflection group $W$. 

\section{Preliminaries}

\subsection{PL spaces and PL structures}
\label{sub:pl_structure}
We remind of some concepts from piecewise linear topology. For more details we refer to \cite{MR0248844} and \cite{MR0350744}. A subset $P\subset \R^n$ is called a \emph{polyhedron} if, for every point $x \in P$, there exists a finite number a simplices contained in $P$ such that their union is a neighborhood of $x$ in $P$.
Every polyhedron $P$ in $\R^n$ is the underlying space of some (locally finite Euclidean) simplicial complex $K$ in $\R^n$ \cite[Lem.~3.5]{MR0248844}. Such a complex is called a \emph{triangulation} of $P$. Conversely, the underlying space of a simplicial complex $K\subset \R^n$ is a polyhedron. A continuous map $f:P\To Q$ between polyhedra $P\subset \R^n$, $Q\subset \R^m$ is called \emph{piecewise linear} (\emph{PL}), if its graph $\{(x,f(x))|x \in P\} \subset \R^{m+n}$ is a polyhedron. 

A \emph{PL chart} $(P,\varphi)$ for a topological space $X$ is an embedding $\varphi: P \To X$ of a compact polyhedron $P$ (i.e. a Euclidean polyhedron in terms of \cite{MR0248844}). Two PL charts $(P,\varphi)$ and $(Q,\phi)$ are said to be \emph{compatible} if $\varphi^{-1}(\phi(Q))$ is a compact polyhedron and $\phi^{-1}\circ\varphi: \varphi^{-1}(\phi (Q)) \To Q$ is piecewise linear. An \emph{atlas} (\emph{base of a PL structure} in terms of \cite{MR0248844}) on $X$ is a family of compatible PL charts for $X$ such that for each point $x \in X$ there is a chart $(P,\varphi)$ for which $\varphi(P)$ is a topological neighborhood of $x$. A \emph{PL structure} on $X$ is a maximal atlas of $X$. A second-countable topological Hausdorff space $X$ endowed with a PL structure is called a \emph{PL space} (cf. \cite[p.~77]{MR0248844}). A PL space $X$ is called a \emph{PL manifold} (\emph{with boundary}) of dimension $n$, if for every point $p\in X$ there exists a chart $(P=\Delta^n, \varphi)$ of $X$ such that $p \in \varphi(\mathrm{Int}(P))$ ($p \in \mathrm{Int}_X(\varphi(P))$) (cf. \cite[p.~79]{MR0248844}). The boundary of a PL $n$-manifold with boundary, i.e. the complement of its PL manifold points, is a PL $(n-1)$-manifold without boundary. For instance, the standard simplex $\Delta^{n}$ is a PL manifold with boundary, a (standard) \emph{PL $n$-ball}. Its boundary $\partial \Delta^{n}$ is a PL manifold without boundary, a (standard) \emph{PL $(n-1)$-sphere}.

Every PL space can be triangulated by a simplicial complex $K\subset \R^{N}$ and can thus be realized as a polyhedron \cite[Lem.~3.5, p.~80]{MR0248844}\cite[Thm.~7.1, p.~53]{MR0488059}. Conversely, an (abstract) locally finite simplicial complex $K$ has a natural PL structure. It is a PL $n$-manifold (with boundary), if and only if the link of every vertex is a PL $(n-1)$-sphere (or a PL $(n-1)$-ball) (cf. \cite[p.~24]{MR0350744}). PL maps and PL homeomorphisms between PL spaces are defined in the usual way via compositions with charts (cf. \cite[p.~83]{MR0248844}). A map between two PL spaces $X$ and $Y$ is a PL homeomorphism if and only if there exist triangulations $K$ and $L$ of $X$ and $Y$ with respect to which $f$ is a simplicial isomorphism between simplical complexes (cf. \cite[p.~84, Thm.~3.6 C]{MR0248844}). 

The open cone $CX$ of a compact PL space $X$ inherits a natural PL structure from $X$. If $X$ is embedded as a polyhedron $P$ in some $\R^N$, then $CX$ is PL homeomorphic to the internal open cone $CX= \R_{\geq 0} \cdot (P+e_{n+1}) \subset \R^{n+1}$. Here $e_{n+1}$ denotes the last canonical basis vector of $\R^{n+1}$. The cone $CX$ is PL homeomorphic to some $\R^n$ if and only if $X$ is PL homeomorphic to the standard PL $(n-1)$-sphere $\partial \Delta^n$. To see the only-if direction of this statement one can triangulate $X$ by a simplicial complex $K$ and extend this triangulation to a triangulation $L$ of $CX$ such that the link of the cone point in $L$ is $K$ \cite[proof of Prop.~2.9]{MR0350744}. Then the statement follows from the remark above on links in simplical complexes that are PL manifolds. 

\subsection{PL quotients}
\label{sub:pl_quotients}

Let $\sim$ be an equivalence relation on a PL space $X$. We would like to know whether $X/\sim$ is a PL space such that the projection map $q: X \To X/\sim$ is PL. More precisely, if there exists a PL space $Y$ and a PL map $f:X\To Y$ that induces a homeomorphism $\overline{f}:X/\sim \To Y$. In general such a pair $(Y,f)$ does not exist (cf. \cite{MR1114628}). However, if such a pair exists, then the following universal property shows that $Y$ is unique up to PL homeomorphism and can thus be considered the quotient of $X$ with respect to $\sim$ in the PL category (cf. \cite{MR1114628}). Note that we identify $X\backslash\sim$ and $Y$ via $\overline{f}$.

\begin{lem}\label{lem:universal_prop}
If $Y'$ is a PL space and $f':X \To Y'$ a PL map such that $x\sim y$ for $x,y\in X$ implies $f'(x)=f'(y)$, then the unique map $g:Y \To Y'$ is PL.
\end{lem}
\begin{proof} Let $x\in X$, $y=q(x)$ and $y'=f'(x)$. We choose charts $(P,\varphi)$, $(Q,\phi)$ and $(Q',\phi')$ about $x$, $y$ and $y'$ that define topological neighborhoods of the respective points. Since $q$ is open and since the image of a compact polyhedron under a PL map is a compact polyhedron \cite[Cor.~2.5]{MR0350744}, we can assume that $q\varphi(P)=\phi(Q)$ and $P=\varphi^{-1}f'^{-1}\phi'(Q')$. In particular, we have $\phi^{-1}g^{-1}\phi'(Q')=Q$ and $\mathrm{graph}(\phi'^{-1}g\phi)=( \phi^{-1}q\times \phi'^{-1}f')(\varphi(P)) \subset Q \times Q'$ is a compact polyhedron, i.e. $\phi'^{-1}g\phi:Q \To Q'$ is a PL map. Now the general facts that a finite union $\bigcup_i P_i$ of compact polyhedra $P_i$ is a compact polyhedron and that a map $f:\bigcup_i P_i \To Q$ is PL if all restrictions $f_{|P_i}$ are PL \cite[p.~5, 1.5 (4)]{MR0350744}, imply the condition for all charts, i.e. the map $g$ is PL.
\end{proof}
A Euclidean vector space $\R^n$ carries a natural PL structure with respect to which it is a PL manifold and with respect to which $\Or_n$ acts by PL homeomorphisms on it. In the following section we realize $\R^n/G$ as a simplicial complex and show that the projection from $\R^n$ to $\R^n/G$ with the induced PL structure is a PL map, i.e. that the quotient $\R^n/G$ is in a natural way a PL space.

\subsection{Admissible triangulations}
\label{sub:ad_triangulations}

Let $G$ be a finite group. A simplicial complex $K$ is called a \emph{$G$-complex}, if $G$ acts simplicially on it. It is called a \emph{regular} $G$-complex, if for each subgroup $H<G$ and each tuple of elements $g_0,g_1,\ldots,g_n \in H$ such that both of the sets $\{v_0,\ldots,v_n\}$ and $\{g_0v_0,\ldots,g_nv_n\}$ describe vertices of a simplex in $K$, there exists an element $g\in H$ such that $gv_i = g_i v_i$ for all $i$ (cf. \cite[Ch.~III, Def.~1.2, p.~116]{MR0413144}). The second barycentric subdivision of a $G$-complex $K$ is always regular and for a regular $G$-complex one can define in a natural way a simplicial complex $K/G$ whose underlying space is homeomorphic to the topological quotient $|K|/G$ \cite[p.~117]{MR0413144}. The vertices of $K/G$ are the $G$-orbits of the vertices of $K$ and a subset of these vertices forms a simplex if and only if there are representatives of these vertices in $K$ that form a simplex in $K$. 

For a finite subgroup $G<\Or_n$ we call a triangulation $K$ of $\R^n$ \emph{admissible (for the action of $G$ on $\R^n$)}, if $K$ is a regular $G$-complex that contains the origin as a vertex. The following lemma shows that admissible triangulations always exist.
\begin{lem}\label{lem:common_subdivision}
For a finite subgroup $G<\Or_n$ there exists a triangulation $K$ of $\R^n$ that is admissible for the action of $G$ on $\R^n$.
\end{lem}
\begin{proof} We start with any triangulation $t:\tilde{K} \To \R^n$ that contains the origin as a vertex and replace it by the common subdivision $t:K \To \R^n$ of the triangulations $g\circ t: |\tilde{K} |\To \R^n$, $g\in G$ as in \cite[Lem.~2.1]{Lange} (cf. \cite[Lemma 67]{Lange_thesis}). The resulting triangulation defines a $G$-complex. Upon passage to the second barycentric subdivision, we can assume that this $G$-complex is regular \cite[p.~117]{MR0413144}.
\end{proof}

Let $K$ be an admissible triangulation for the action of $G<\Or_n$ on $\R^n$. Then $Y=K/G$ is a simplicial complex and the projection $K\To K/G$ maps simplices linearly onto simplices. In particular, $Y$ is a PL space and the projection $K\To K/G$ is PL, i.e. $\R^n/G$ is in a natural way a PL space (cf. Section \ref{sub:pl_quotients}). The link of the origin in $K$ is also a regular $G$-complex and its quotient by $G$ is simplicially isomorphic to the link of the origin in $K/G$. Hence, the PL space $\R^n/G$ is a PL manifold (with boundary) if and only if this link in $K/G$ is a PL $(n-1)$-sphere (or a PL $(n-1)$-ball) (cf. Section \ref{sub:pl_structure}). Radially projecting the link of the origin in $K$ to the unit sphere $S^{n-1}$ defines a PL structure on $S^{n-1}$ and induces a PL structure on $S^{n-1}/G$. We call triangulations and PL structures of $S^{n-1}$ and $S^{n-1}/G$ that arise in this way \emph{admissible (for the action of $G$ on $S^{n-1}$)}. Two admissible PL structures on $S^{n-1}/G$ need not be identical but are PL homeomorphic (cf. \cite[pp.~20-21]{MR0350744}, ``pseudoradial projection''). Hence, the question if $\R^n/G$ is a PL manifold is equivalent to the question if $S^{n-1}/G$ is a PL sphere with respect to one and then any admissible PL structure. The following lemma gives a necessary condition for this to hold. For details we refer to \cite{Lange_thesis}.

\begin{lem} \label{lem:PL_isotropy_condition}
Let $G< \Or_n$ be a finite subgroup and suppose $S^{n-1}$ is triangulated by a simplicial complex $K$ in an admissible way for the action of $G$. Then $S^{n-1}/G$ is a PL manifold (with boundary) with respect to the induced PL structure if and only if for every vertex $v$ of $K$ the quotient space $T_v S^{n-1}/G_v$ is a PL manifold (with boundary). In particular, this is the case if $\R^n/G$ is a PL manifold (with boundary).
\end{lem}

\section{The only-if direction}
\label{sec:only-if}

Let us show by induction that a finite subgroup $G<\Or_n$ is a reflection-rotation group, if the quotient space $\R^{n}/G$ is a PL manifold with boundary and that in this case $G$ contains a reflection if and only if the boundary of $\R^{n}/G$ is nonempty. For $n\leq 2$ the claim is trivially true. Assume it holds for some $n\geq 2$ and let $G<\Or_{n+1}$ be a finite subgroup such that $\R^{n+1}/G$ is a PL manifold with boundary. Then all isotropy groups $G_v$ for $v\neq 0$ are reflection-rotation groups by Lemma \ref{lem:PL_isotropy_condition} and the induction assumption. Let $G_{rr}\triangleleft G$ be the reflection-rotation group generated by all of them and let $v \in S^{n}$. Then the inclusions
\[
	G_v = (G_v)_{rr}	\subseteq (G_{rr})_v \subseteq G_v
\]
imply that $G_v=(G_{rr})_v$ for all $v \in S^{n}$. This means that the action of $G/G_{rr}$ on $S^{n}/G_{rr}$ is free. Because of $n\geq 2$ the quotient space $S^{n}/G$ is simply connected by assumption and thus we conclude that $G$ and $G_{rr}$ coincide, i.e. that $G$ is a reflection-rotation group. Again by the induction assumption, $G_{rr}$ contains a reflection if and only if the boundary of $S^{n}/G_{rr}$ is nonempty. Hence, $G$ contains a reflection if and only if the boundary of $\R^n/G$ is nonempty and the claim follows by induction.

\section{Methods for the if direction}

Before proving the if direction of our main result, we provide several methods that allow us to identify the quotient spaces in question.

\subsection{PL linearization principle}
\label{sub:linearization_principle}
The idea of the PL linearization principle is to divide the determination of the PL quotient $\R^n/G$ for a finite subgroup $G < \Or_n$ into several steps. Let $H \triangleleft G$ be a normal subgroup and assume there exists a PL homeomorphism $F: \R^{n}/H \To \R^{n}$ and a homomorphism $r: G \To \Or_n$ with kernel $H$ such that the left square in the following diagram commutes
\[
	\begin{xy}
		\xymatrix
		{
		   G \times \R^{n}/H \ar[r] \ar[d]_{r\times F}  & \R^{n}/H \ar@{->>}[r] \ar[d]_{F}  & \R^{n}/G  \ar@{-->}[d]_{\tilde{F}}  \\
		   r(G) \ar[r] \times \R^{n} & \R^{n} \ar@{->>}[r]   & \R^{n}/r(G)  
		}
	\end{xy}
\]	
Then we say that the PL linearization principle can be applied to the groups $H\triangleleft G$. In this case $F$ induces a PL homeomorphism $\tilde{F}: \R^{n}/G \To \R^{n}/r(G)$ due to Lemma \ref{lem:universal_prop}. This reduces the determination of $\R^{n}/G$ to the determination of $\R^{n}/r(G)$ and one might look for a suitable normal subgroup of $r(G)$ in order to apply the PL linearization principle again. If the PL linearization principle can be applied to $H\triangleleft G$ and to $\tilde{H} \triangleleft r(G)$, then it can also directly be applied to $r^{-1}(\tilde{H})\triangleleft G$. 

\begin{lem} Suppose that the PL linearization principle can be applied to groups $H\triangleleft G$. If $g \in G$ is a reflection (rotation), then so is $r(g)$. In particular, if $G$ is generated by reflections and rotations, then so is $r(G)$.
\end{lem}

The PL linearization principle can be established by describing a homeomorphism $f: S^{n-1}/H \To S^{n-1}$ and a homomorphism $r: G \To \Or_n$ with kernel $H$ such that the following square commutes and such that the PL structure on $S^{n-1}$ induced by an admissible PL structure on $S^{n-1}/H$ via $f$ is admissible for the linearized action of $r(G)$ on $S^{n-1}$
\[
	\begin{xy}
		\xymatrix
		{
		   G \times S^{n-1}/H \ar[r] \ar[d]_{r\times f}  & S^{n-1}/H  \ar[d]_{f}    \\
		   r(G) \ar[r] \times S^{n-1} & S^{n-1}   
		}
	\end{xy}
\]
In fact, in this case we can take the open cone at each site and extend the PL homeomorphism $f$ linearly to a PL homeomorphism $F: \R^{n}/H \To \R^{n}$ which makes the first diagram above commute.


\subsection{Uniqueness of compatible PL structures.}
\label{sub:approx}

A map $g:P \To M$ from a polyhedron $P$ to a smooth manifold $M$ is called \emph{piecewise differentiable} or \emph{PD}, if $P$ admits a triangulation such that the restriction of $f$ to each simplex in this triangulation is smooth. It is called \emph{PD homeomorphism}, if it is in addition a homeomorphism and its restriction to each simplex has injective differential at each point. If $g$ is a PD homeomorphism, then, due to a theorem by Whitehead \cite{Whitehead}, the polyhedron $P$ is a PL manifold. Moreover, such a polyhedron exists and is unique up to PL homeomorphisms. According to a result of Illman these statements also hold equivariantly for a finite group acting smoothly on $M$ \cite{MR0500993}. In particular, if $M$ is a smooth manifold on which a finite group $G$ acts smoothly and $g_i:P_i \To M$, $i=1,2$, are two PD homeomorphisms such that the induced actions of $G$ on the polyhedra $P_1$ and $P_2$ are PL, then there exists a $G$-equivariant PL homeomorphism between $P_1$ and $P_2$.

Assume we have a PL $(n-1)$-sphere $P_1$ on which a finite group $G$ acts by PL homeomorphisms, a PD homeomorphism $g: P_1 \To S^{n-1}$ and a group homomorphism $r:G\To \Or_n$ such that the following diagram commutes
\[
	\begin{xy}
		\xymatrix
		{
		   G \times P_1 \ar[r] \ar[d]_{r\times g}  & P_1  \ar[d]_{g}    \\
		   r(G) \ar[r] \times S^{n-1} & S^{n-1}   
		}
	\end{xy}
\]
Then, by Section \ref{sub:pl_quotients} there exists an admissible triangulation of $S^{n-1}$ by a polyhedron $P_2$ with respect to the action of $r(G)$ and this triangulation defines a PD homeomorphism from the polyhedron to $S^{n-1}$. Therefore $P_1$ and $P_2$ are $G$-equivariantly PL homeomorphic as explained above. In other words, we can replace $g$ by another $G$-equivariant homeomorphism $f$ which is in addition piecewise linear with respect to an admissible PL structure for the action of $r(G)$ on $S^{n-1}$.

\subsection{Generalized Poincar\'e conjecture}
\label{sub:poincare}

The generalized Poincar\'e conjecture holds in the following version. 
Note that a closed simply connected topological manifold $M$ with $H_*(M;\Z)=H_*(S^n;\Z)$ is homotopy equivalent to an $n$-sphere (cf. \cite[Thm.~7.5.9, p.~399]{MR0210112} and \cite[Thm.~4.5, p.~346]{MR1867354}).

\begin{thm}\label{thm:poincare_pl_manifold}
For $n\neq 4$ a closed simply connected PL manifold $M$ with $H_*(M;\Z)=H_*(S^n;\Z)$ is PL homeomorphic to a standard PL $n$-sphere.
\end{thm}

For $n=1,2$ the statement has long be known by the classification of manifolds in that dimensions. For $n\geq 6$ it follows from the PL h-cobordism theorem (c.f. \cite[Thm.~A, p.~17]{MR0350744}). For $n\leq 5$ every PL manifold admits a compatible smooth structure \cite{MR0415630} and thus, according to the uniqueness part of Whitehead's theorem \cite{Whitehead,MR0198479}, the statement can be reduced to the respective statement in the smooth category. For $n=5$ the smooth version of the generalized Poincar\'e conjecture follows from the smooth h-cobordism theorem combined with the fact that every closed, smooth $5$-manifold homotopy equivalent to $S^5$ bounds a smooth, compact, contractible $6$-manifold \cite{MR0148075} (cf. \cite{MR0190942} for more details). Finally, for $n=3$ the smooth Poincar\'e conjecture follows from Perelman's work \cite{Per02,Per03a,Per03b} (cf. \cite{MorganTian,KleLo06}).

We would like to apply the statement of Theorem \ref{thm:poincare_pl_manifold} in the following situation. For a rotation group $G<\SOr_n$ the quotient $S^{n-1}/G$ is simply connected unless $n\leq 2$ by the following lemma (cf. \cite{Armstrong}, where the statement is proven in greater generality).

\begin{lem} \label{lem:simply_con_quotient} Let $G<\SOr_{n}$ with $n\geq 3$ be a finite subgroup generated by elements that fix some point in $S^{n-1}$. Then the quotient space $S^{n-1}/G$ is simply connected. 
\end{lem}

Suppose $S^{n-1}$ and $S^{n-1}/G$ are equipped with PL structures that are admissible for the action of $G$ (cf. Section \ref{sub:ad_triangulations}). According to Lemma \ref{lem:PL_isotropy_condition}, in order to show that $S^{n-1}/G$ is a closed PL manifold it suffices to check that for any point $p \in S^{n-1}$ the isotropy group $G_p$ is a rotation group acting in $T_p S^{n-1}=\R^{n-1}$ such that the quotient space $\R^{n-1}/G_p$ is a PL manifold. The condition on the homology groups of $S^{n-1}/G$ can be verified as follows. 

\begin{lem} \label{lem:homology_condition} Let a finite subgroup $G<\SOr_{n}$ and some $i \in \Z_{\geq 0}$ be given. If there are subgroups $H$ of $G$ with coprime indices and $H_i(S^{n-1}/H;\Z)=0$, then  $H_i(S^{n-1}/G;\Z)=0$. In particular, the conclusion holds if the subgroups have coprime indices and the quotients $\R^n/H$ are homeomorphic to $\R^n$.
\end{lem}
\begin{proof}
We choose an admissible triangulation $K$ of $S^{n-1}$ for the action of $G$ and work with simplicial homology over $\Z$ (in the following the coefficient ring is omitted and understood to be $\Z$). Let $H<G$ be a subgroup with $H_i(S^{n-1}/H)=0$. For an $i$-cycle $c\in Z_i(K/G)$ there exists an $i$-cycle $c' \in  Z_i(K/H)$ such that $\pi(c')= [G:H] \cdot c$ where $\pi_{G/H}:K/H \To K/G$ is the natural simplicial projection (e.g. $c'=\mu_{G/H}(c)$ in the notation of \cite[pp.~118-121]{MR0413144}). The $(i+1)$-chain $a\in C_{i+1}(K/H)$ with $\partial a= c'$ satisfies $\partial \pi(a)= \pi(\partial a)= [G:H] \cdot c$ and thus $0=(\pi_{G/H})_*(\mu_{G/H})_*([c])=[G:H] \cdot [c]$ in $H_i(S^{n-1}/G)$ (the induced map $(\mu_{G/H})_*: H_*(K/G)\To H_*(K/H)$ is called \emph{transfer}, cf. \cite[Ch. III. 2., pp.~118-121]{MR0413144}). Hence, if there are subgroups $H$ of $G$ with coprime indices and $H_i(S^{n-1}/H)=0$, then it follows that $H_i(S^{n-1}/G)=0$. The assumption on $H_i(S^{n-1}/H)$ is in particular fulfilled, if $\R^n/H$ is homeomorphic to $\R^n$, since this implies $H_*(S^{n-1}/H)=H_*(S^{n-1})$ (cf. \cite[p.~117]{MR1867354}).
\end{proof}

\subsection{Chevalley's theorem}\label{sub:alg_inv}

For a unitary reflection group $G<\mathrm{U}_n$ the following theorem due to Chevalley holds \cite[Thm. 3.20, p.~48]{MR2542964}.

\begin{thm} The algebra of invariants of a finite unitary reflection group $G <\mathrm{U}_n$ is generated by $n$ algebraically independent homogenous polynomials.
\end{thm}

For $n$ such generators $f_1,\ldots,f_n$ of $\Co[z_1,\ldots,z_n]^G$ we will see in Section \ref{sub:unit_reflection_groups_hom} that the map
\[
\begin{array}{lccc}
   f=[f_1,\ldots,f_n]: 	& \Co^n & \longrightarrow  				& \Co^n \\
   					& v		& \longmapsto &  (f_1(v),\ldots,f_n(v))
\end{array}
\]
descends to a homeomorphism $\overline{f}: \Co^n/G \To \Co^n$ and, moreover, that $\R^{2n}/G$ is in fact PL homeomorphic to $\R^{2n}$.

\subsection{The fundamental domain of a group}\label{sub:fundamental_domain}
For a finite subgroup $G < \Or_n$ there exists a vector $v_0 \in \R^n$ such that $gv_0 \neq v_0$ for all $g \in G\backslash \{e\}$. For such a vector $v_0$ the set
\[
		\Lambda = \bigcap_{g\in G} \{v\in \R^n| (v,v_0)\geq (v,gv_0)\}
\]
is a fundamental domain for the group $G$, i.e. the $G$-translates of $\Lambda$ cover $\R^n$ and the union $\bigcup_{g\in G} g \mathring{\Lambda}$ is disjoint. It inherits a subspace topology and PL structure from $\R^n$ and the quotient space $\R^n/G$ with its quotient topology and PL structure (cf. Sections \ref{sub:pl_quotients} and \ref{sub:ad_triangulations}) can be obtained from $\Lambda$ by identifying certain points on the boundary, namely those which belong to the same orbit of $G$.

\subsection{Gluing construction}

We need the following elementary gluing result for PL balls. The lemma states that ``twisted'' PL spheres are standard PL spheres. Its proof is a direct consequence from the fact that a PL homeomorphism $f: \partial \Delta^n \To \partial \Delta^n$ can be linearly extended in a radial direction to a PL homeomorphism $\overline{f}:\Delta^n\To \Delta^n$ \cite[Lem.~1.10, p.~8]{MR0350744}.
\begin{lem} \label{lem:gluing_pl_balls}
Suppose $B_1^n$ and $B_2^n$ are PL balls and $\varphi: \partial B_1^n \To \partial B_2^n$ is a PL homeomorphism. Then the space $B_1^n \cup_{\varphi} B_2^n$ obtained by gluing $B_1^n$ and $B_2^n$ together along their boundary via $\varphi$ is a PL $n$-sphere.
\end{lem}

\subsection{Collapsing} \label{sub:collapse}
Let $K$ be a simplicial complex and let $\sigma, \tau \in K$ be simplices such that
\begin{compactenum}
\item $\tau < \sigma$, i.e. $\tau$ is a proper face of $\sigma$,
\item $\sigma$ is a maximal simplex in $K$ and $\tau$ is not contained in any other maximal simplex of $K$,
\end{compactenum}
then $\tau$ is called a \emph{free face} of $K$. A \emph{simplicial collaps} of $K$ defined by $\tau < \sigma$ as above is the removal of all simplices $\rho$ of $K$ with $\tau \leq \rho \leq \sigma$. We say that $K$ collapses onto a subcomplex $L$ of $K$ if there exists a finite sequence of collapses leading from $K$ to $L$. A simplicial complex that collapses onto a point is called \emph{collapsible}. Being collapsible to a subcomplex is a PL property, i.e. it does not depend on a specific triangulation (cf. \cite[p.~39]{MR0350744}). In our proof we will apply the following characterization (cf. \cite[Cor.~3.28, p.~41]{MR0350744})

\begin{lem}\label{lem:collapse} A collapsible PL $n$-manifold (with or without boundary) is a PL $n$-ball.
\end{lem}

In order to be able to apply this characterization, we will need the following lemma. 

\begin{lem}\label{lem:induced_collapse} Let $p: K \To \tilde{K}$ be a simplicial surjection between finite simplicial complexes $K$ and $\tilde{K}$ that maps simplices of $K$ homeomorphically onto simplices of $\tilde{K}$. Suppose further that $L$ is a subcomplex of $K$ such that $p$ restricts to a bijection $p: K \backslash L \To \tilde{K} \backslash p(L)$. If $K$ collapses onto $L$, then $\tilde{K}$ collapses onto $p(L)$
\end{lem}
\begin{proof} Let  $\tau < \sigma \in K$ with $\tau, \sigma \notin L$ and suppose that $\tau < \sigma$ defines a simplicial collaps of $K$. By assumption on $p$ we have $p(\tau), p(\sigma) \notin p(L)$ and thus the claim follows inductively, if we can show that $p(\tau) < p(\sigma) \in \tilde{K}$ defines a simplicial collapse of $\tilde{K}$. But again, by assumption on $p$, and because of $\tau, \sigma \notin L$ it is clear that $p(\sigma)$ is a maximal simplex of $\tilde{K}$ and that $p(\tau)$ is not contained in any other maximal simplex of $\tilde{K}$. 
\end{proof}

\section{Proof of the main result}

In this section we prove the if direction of our main result by verifying its conclusion for all reflection-rotation groups. The proof is structured as follows. For each reflection-rotation group $G$ we either prove the conclusion of our main result directly or we reduce such a proof to the respective claim on reflection-rotation groups of lower order via the PL linearization principle (cf. Appendix \ref{app:class-rr-groups} for a summary of the classification result). In doing this we will need to show that for each pair $M\triangleleft G_{rr}$ of an irreducible reflection-rotation group $G_{rr}$ that contains a reflection and a nontrivial normal rotation group $M \triangleleft G_{rr}$ such that $G_{rr}$ is generated by the reflections it contains and by $M$, there exists a nontrivial rotation group $H \triangleleft M$ normalized by $G_{rr}$ such that the PL linearization principle can be applied to the groups $H\triangleleft G_{rr}$ (cf. Section \ref{sub:red_pse_gro}). All such pairs $M\triangleleft G_{rr}$ are listed in Theorem \ref{thm:reducible_pairs}. In each case we will either show this property directly or reduce it to proving our main result for reflection-rotation groups of order less than $G_{rr}$. Once we have treated all the cases, the if direction of our main result follows by induction. References to the sections in which the respective cases are treated can be found in the appendix.

\subsection{Real reflection groups}
\label{sub:reflection_groups_hom}

The fundamental domain $\Lambda$ of a reflection group $W<\Or_n$ acting on $S^{n-1}$ is a spherical simplex \cite[Thm. 4, p.~595]{MR1503182}. Let $\Wp$ be the orientation preserving subgroup of $W$ and, if there exists some rotation $h\in \Or_n\backslash W$ that normalizes $W$, set $\Wt=\left\langle W,h\right\rangle$ and $\Ws=\left\langle \Wp,h\right\rangle$. Choose an admissible triangulation for the action of $W$ (and hence of $W^{+}$) on $S^{n-1}$ that refines the triangulation of $S^{n-1}$ by the fundamental domains of $W$. Then the quotient space $S^{n-1}/W$ is a PL ball, namely the fundamental domain $\Lambda$ of $W$, and the quotient space $S^{n-1}/\Wp$ can be obtained by gluing together two copies of $\Lambda$ along their boundary, i.e. the resulting space is a PL sphere by Lemma \ref{lem:gluing_pl_balls}. Moreover, a coset $\overline{s}$ of a reflection $s\in W$ interchanges the two copies. Therefore the PL linearization principle can be applied to the groups $\Wp\triangleleft W$. If $h$ exists, then its action on $S^{n-1}/\Wp$ commutes with the action of a reflection $s \in W$ on $S^{n-1}/\Wp$, since $h$ normalizes $W$ by assumption. In particular, the quotient space $S^{n-1}/\Ws$ can be realized as the suspension of $\partial \Lambda / \overline{h}$ whose cone points are interchanged by $\overline{s}$. Therefore, it is clear that $S^{n-1}/\Ws$ is a PL sphere, that $S^{n-1}/\Wt$ is a PL ball and that the PL linearization principle can be applied to the groups $\Wp \triangleleft \Wt$ and $\Ws \triangleleft \Wt$. Hence, we have proven

\begin{lem}\label{lem:main_refl} In the notation used above our main theorem holds for groups of type $W$, $\Wp$, $\Ws$, $\Wt$ and the PL linearization principle can be applied to $\Wp \triangleleft W$, $\Wp \triangleleft \Wt$ and $\Ws \triangleleft \Wt$. In particular, it can be applied to $\Wp(\D_n) \triangleleft W(\BC_n)$.
\end{lem}

\subsection{Reflection groups induced by unitary reflection groups} \label{sub:unit_reflection_groups_hom}

For a unitary reflection group $G<\mathrm{U}_n$ we choose $n$ algebraically independent homogenous generators $f_1,\ldots,f_n \in \Co[z_1,\ldots,z_n]^G$ given by Chevalley's theorem (cf. Section \ref{sub:alg_inv}). The continuous map
\[
\begin{array}{lccc}
   f=[f_1,\ldots,f_n]: 	& \Co^n & \longrightarrow  				& \Co^n \\
   					& v		& \longmapsto &  (f_1(v),\ldots,f_n(v))
\end{array}
\]
descends to a continuous map $\overline{f}:\Co^n/G \To \Co^n$. The map $\overline{f}$ is injective, since the algebra of invariants of $G$ separates its orbits \cite[Thm. 3.5, p.~41]{MR2542964}, and also onto \cite[Thm.~3.15, p.~45]{MR2542964}. Moreover, since $\Co[z_1,\ldots,z_n]$ is a finitely generated $\Co[z_1,\ldots,z_n]^G$-module \cite[Thm.~1.3, p.~478]{MR526968}, the map $f$ is a finite and hence proper morphism of complex affine varieties \cite[6.1.11, 5.5.3]{MR0163909}. Therefore, the map $f$ is also proper with respect to the usual topology in the sense of \cite[Ch.~1, $\S$10, no.~1, Def.~1]{BourbakiTop} by \cite[Ch.~XII., Prop.~3.2]{MR2017446}. In particular, the map $f$ is closed. Consequently $\overline{f}$ is a homeomorphism and thus $\R^{2n}/G$ and $\R^{2n}$ are homeomorphic where $G$ is regarded as a real rotation group.

Continuity of $\overline{f}^{-1}$ can alternatively be shown by induction as follows. According to a theorem by Steinberg isotropy groups of unitary reflection groups are again unitary reflection groups \cite[Thm.~1.5, p.~394]{MR0167535} (cf. \cite[Thm. 9.44, p.~186]{MR2542964} and \cite{MR2052515}). Hence, it follows by induction that $\Co^n/G - \{\overline{0}\}$ is a topological manifold, where $\overline{0}$ is the coset of $0\in \Co^n$ in $\Co^n/G$. Using the domain invariance theorem it is then not difficult to conclude that $\overline{f}^{-1}$ is continuous (cf. \cite{Lange_thesis} for more details).

Finally, since isotropy groups of unitary reflection groups are again unitary reflection groups by Steinbergs theorem, it follows by induction as explained in Section \ref{sub:poincare} that the PL quotient $\R^{2n}/G$ for a unitary reflection group $G$ regarded as a rotation group is PL homeomorphic to $\R^{2n}$, i.e. we have

\begin{lem}\label{lem:main_unit} Our main theorem holds for unitary reflection groups considered as real groups.
\end{lem}

\subsection{Reflection-rotation groups in dimension up to four}
\label{sub:dim_four_groups_hom}

Groups up to dimension three can be easily treated by hand. For instance, all finite subgroups of $\SOr_2$ and $\SOr_3$ are orientation preserving subgroups of reflection groups which have been treated in Section \ref{sub:reflection_groups_hom}. In dimension four a classification based proof becomes rather long-winded and cumbersome (cf. \cite[\S 3]{Mikhailova}).

The following approach dispenses with the classification in dimensions up to four. The proof is based on induction on the dimension. For simplicity let us suppose that we have already treated the cases $n<4$, i.e. we formulate the proof for $n=4$. The same arguments apply in dimenions $n<4$.
All isotropy groups of a rotation group $G<\SOr_4$ are again rotation groups. Therefore, the quotient $S^3/G$ is a closed simply connected PL manifold with respect to a PL structure induced by $K/G$ where $K$ is an admissible triangulation for the action of $G$ on $S^3$ (cf. Lemma \ref{lem:PL_isotropy_condition}). Hence, by the PL version of the Poincar\'e conjecture (cf. Section \ref{sub:poincare}), the PL quotient $S^3/G$ is PL homeomorphic to the standard PL $3$-sphere and $\R^4/G$ is PL homeomorphic to $\R^4$, i.e. we have

\begin{lem}\label{lem:main_uptofour} Our main theorem holds for rotation groups in dimension up to four.
\end{lem}

Now let $G<\Or_4$ be a finite group and suppose that $H \triangleleft G$ is a rotation group (again, the case of lower dimensions works analogous). We endow $S^3$ with an admissible PL structure for the action of $G$ (and hence $H$). According to \cite{Lange} (cf.  \cite{Lange_thesis} and \cite[p.~208]{MR1435975}) the quotient $S^3/H$ admits a smooth structure such that the identity map of $S^3/H$ is a PD homeomorphism and such that $G/H$ acts smoothly on it. Since smoothings of PL manifolds in dimension three are unique up to diffeomorphism \cite[Thm.~3.10.9]{MR1435975}, this action is smoothly conjugate to a smooth action of $G/H$ on the standard sphere $S^3$ and by \cite[Thm.~E]{MR2491658} it is thus smoothly conjugate to a linear action on $S^3$. Therefore, we have a PD homeomorphism $g:S^3/H \To S^3$ and a homomorphism $r:G\To \SOr_4$ such that the following diagram commutes
\[
	\begin{xy}
		\xymatrix
		{
		   G \times S^3/H \ar[r] \ar[d]_{r\times g}  & S^3/H  \ar[d]_{g}    \\
		   r(G) \ar[r] \times S^3 & S^3   
		}
	\end{xy}
\]
According to Section \ref{sub:approx} we can replace $g$ by a PD homeomorphism $f$ such that the induced PL structure on $S^3$ is admissible with respect to the action of $r(G)$. Therefore, the PL linearization principle can be applied to the groups $H\triangleleft G$ (cf. Section \ref{sub:linearization_principle}). In particular, taking $G$ as a reflection-rotation group and $H$ as its orientation preserving subgroup proves our main theorem for all reflection-rotation groups up to dimension four. Summarizing we have

\begin{lem}\label{lem:main_uptofour_sum} Our main theorem holds for reflection-rotation groups in dimension up to four. If $H <\Or_n$, $n\leq 4$, is a rotation group normalized by a finite group $G<\Or_n$, then the PL linearization principle can be applied to the groups $H \triangleleft G$.
\end{lem}

Note that in principle the usage of the Poincar\'e conjecture can be avoided by applying other means such as the algebra of polynomial invariants or explicit constructions of fundamental domains (cf. \cite{Mikhailova} and Section \ref{sub:exceptional_groups_hom} for illustrations of these methods). However, proofs along such lines are cumbersome and so we do not refrain from using the Poincar\'e conjecture as a convenient tool.

\subsection{Reducible reflection-rotation groups}
\label{sub:red_pse_gro}

Now let $G<\Or_n$ be a reducible reflection-rotation group and let $\R^n=V_1+\ldots +V_k$ be a decomposition into irreducible components. Let $H_i \triangleleft G$ be the normal subgroup generated by rotations that only act in $V_i$ (i.e. by rotations of the first kind in $V_i$ in terms of \cite{LaMik}) and let $G_i$ be the projection of $G$ to the $i$-th factor. We can assume that $H_i \neq G_i$ because otherwise $H_i$ splits of as a direct factor. The pairs $H_i \triangleleft G_i$ that occur in this way are classified in  \cite[Thm.~3]{LaMik}. It is shown that this classification amounts to a classification of pairs $M \triangleleft G_{rr}$ of an irreducible reflection-rotation group $G_{rr}$ that contains a reflection and a normal subgroup $M \triangleleft G_{rr}$ generated by rotations such that $G_{rr}$ is generated by its reflections and by $M$ (cf. the remark following \cite[Thm.~3]{LaMik}). All such pairs with nontrivial $M$ are listed in Theorem \ref{thm:reducible_pairs} in the appendix.

Suppose there is some $i \in \{1,\ldots, k\}$ and some nontrivial rotation group $H<H_i$ normalized by $G_i$ such that the PL linearization principle can be applied to $H\triangleleft G_i$. Then the PL linearization principle can also be applied to $H\triangleleft G$. Hence, the following lemma holds (cf. Section \ref{sub:linearization_principle}).

\begin{lem}\label{lem:redu_red} Let $G$ be a reducible reflection-rotation group and suppose that some $H_i$ is nontrivial. Suppose further that our main theorem holds for all reflection-rotation groups of smaller order than $G$. If for each pair $M\triangleleft G_{rr}$ occurring in Theorem \ref{thm:reducible_pairs} there exists some nontrivial rotation group $H\triangleleft M$ normalized by $G_{rr}$ such that the PL linearization principle can be applied to the groups $H\triangleleft G_{rr}$, then our main theorem holds for the group $G$.
\end{lem}

Note that it is necessary to verify the assumption on the pair $M\triangleleft G_{rr}$ in each case of Theorem \ref{thm:reducible_pairs}. In fact, given such a pair $M\triangleleft G_{rr}$ a reducible rotation group $G$ can be constructed with two irreducible components and with $H_1=H_2=M$, $G_1=G_2=G_{rr}$ (cf. \cite[Thm.~4]{LaMik}). In the course of the proof, in each case of Theorem \ref{thm:reducible_pairs} we either verify the condition on $M\triangleleft G_{rr}$ directly or reduce such a proof to showing our main result for rotation groups of order less than $G_{rr}$. A reference to the lemma in which we do this for a specific pair $M\triangleleft G_{rr}$  can also be found in Theorem \ref{thm:reducible_pairs}.

The only case in which we cannot apply Lemma \ref{lem:redu_red} is, in the notation above, when all the $H_i$ are trivial, i.e. when there are no rotations in $G$ that act in a single $V_i$ factor. Suppose this is the case. If the group $G$ does not split as a product of nontrivial factors of lower order, it is either a reflection group of type $\A_1$, a rotation group of type $\Wp(\A_1\times \dots \times \A_1)$ or a rotation group of the form
\[
		\Delta_{\varphi}(W\times W):=\{(g,\varphi(g)) \in W\times W | g \in W\}< \Or_m \times \Or_m 
\]
for some reflection group $W<\Or_m $ and some isomorphism $\varphi:W\To W$ that maps reflections onto reflections \cite[Thm.~4, Cor.~64]{LaMik}. The first two cases are treated in Lemma \ref{lem:main_refl}. For $m<3$ the third case is treated in Lemma \ref{lem:main_uptofour_sum}. If all labels of the Coxeter graph of $W$ lie in $\{3,4,6\}$, then every automorphism of $W$ that maps reflections onto reflections can be realized through conjugation by an orthogonal transformation in the normalizer of $W$ in $\Or_m $ \cite[Cor. 19, p.~7]{MR1997410}. In this case the quotient $\R^{2m}/\Delta_{\varphi}(W\times W)$ is PL homeomorphic to $\R^{2m}/\Delta_{\mathrm{id}}(W\times W)$. Hence this case is subject of Lemma \ref{lem:main_unit}, since $\Delta = \Delta_{\mathrm{id}}(W\times W) <W\times W$ preserves the complex structure
\[
	\begin{array}{ccl}
		 	J= \left(
		  \begin{array}{cc}
		    0 & 1_m  \\
		    -1_m & 0 \\
		  \end{array}\right)
  \end{array}	.
\]
and can thus be regarded as a unitary reflection group acting on $\Co^m$. The only remaining cases are $W=W(\Hn_3)$ and $W=W(\Hn_4)$ (cf. Theorem \ref{thm:reducible_pairs} or \cite[Thm.~3]{LaMik}) and indeed, in these cases there exist outer automorphisms of $W$ that map reflections onto reflections but cannot be realized through conjugation in $\Or_m $ (cf. \cite[pp.~31-32]{Franszen}). Note that the argument in \cite{Mikhailova} breaks down for groups $\Delta_{\varphi}(W\times W)$ for which $\varphi$ cannot be realized through conjugation since the proof of \cite[Thm.~1.2]{Mikhailova} does not work in this case. Summarizing, we have

\begin{lem}\label{lem:redu_red2} In the notation above, let $G$ be a reducible reflection-rotation group such that all the $H_i$ are trivial. Suppose that our main theorem holds for all reflection-rotation groups of smaller order than $G$. If $G$ is different from $\Delta_{\varphi}(W\times W)$ for $W$ of type $\Hn_3$ and $\Hn_4$, then our main theorem holds for the group $G$.
\end{lem}

The two exceptional cases excluded in the lemma are treated in Section \ref{sub:exceptional_groups_hom}. 

\subsection{Monomial reflection-rotation groups}
\label{sub:monomial_groups_hom}

Let $D(n)$ be the diagonal subgroup of $\Or_n$ and let $D^+(n)$ be its orientation preserving subgroup. For a permutation group $H<\sym_n$ consider the monomial groups $M=D^+(n) \rtimes H <\SOr_n$ and $\Mt=D(n) \rtimes H <\Or_n$. The subgroup $H<\Mt$ leaves the spherical simplex $\Lambda=\{x\in S^{n-1}|x_i\geq 0, i=1,\ldots,n\}$ invariant and acts on its boundary $\partial \Lambda$. Choose an admissible triangulation for the action of $\Mt$ on $S^{n-1}$ that refines the triangulation of $S^{n-1}$ by the $D(n)$ translates of $\Lambda$. The quotient $S^{n-1}/D^+(n)$ of $S^{n-1}$ by the diagonal subgroup $D^+(n)$ can be obtained by gluing together two copies of $\Lambda$ along their boundaries. It can be realized as the suspension of $\partial \Lambda$ whose cone points are interchanged by the reflections in $D(n)$. Therefore the PL linearization principle can be applied to the groups $D^+(n)\triangleleft \Mt$. In particular, we have (note that $D(\Wp(\BC_n))=D(\Wp(\D_n))=D^+(n)$)

\begin{lem}\label{lem:line_refl} The PL linearization principle can be applied to the groups $D(\Wp(\BC_n))\triangleleft W(\BC_n)$ and thus also to the groups $D(\Wp(\D_n))\triangleleft W(\D_n)$.
\end{lem}
Note that with respect to the constructed linearization the reflection in $\Mt/D^+(n)$ acts in a $1$-dimensional subspace orthogonal to a subspace $\R^{n-1}$ in which $H\cong M/D^+(n)$ acts. Now suppose that $M$ is a rotation group. Then the linearization of $H$ acts as a rotation group (cf. Section \ref{sub:linearization_principle}). With these observations we obtain
\begin{lem}\label{lem:monom} In the notation above, suppose that $M$ is a rotation group and that our main theorem holds for all reflection-rotation groups of smaller order than $M$. Then our main theorem holds for the reflection-rotation groups $M$ and $\Mt$ and the PL linearization principle can be applied to the pairs $M\triangleleft \Mt$. In particular, this applies in the cases of $M=M_5, M_6,M_7, M_8,M(\D_n)$ in Theorem \ref{thm:clas_red_psg}, $(v)$, $(a)$.
\end{lem}
The exceptional monomial rotation groups $M(\Qq_7)=M^p_7<\SOr_7$ and $M(\Qq_8)=M^p_8<\SOr_8$ (cf. Theorem \ref{thm:clas_red_psg}, $(v)$, $(a)$) are treated in Section \ref{sub:exceptional_groups_hom}.

\subsection{Imprimitive reflection-rotation groups}\label{sub:imprimitive_rot_groups}
Let $\mu_m <  \Co^*$ be the cyclic subgroup of $m$-th roots of unity. For a factor $p$ of $m$ set
\[
		A(m,p,n) := \left\{(\theta_1,\ldots,\theta_n)\in \mu_m^n | (\theta_1\dots \theta_n)^{m/p}=1 \right\}.
\]
Let $G(m,p,n)$ be the semidirect product of $A(m,p,n)$ with the symmetric group $\sym_n$. Then the natural realization of $G(m,p,n)$ in the unitary group $\mathrm{U}_n$ is an imprimitive unitary reflection group (cf. \cite[Ch. 2, p.~25]{MR2542964}). In the following we consider $G(m,p,n)$ as a real rotation group in $\SOr_{2n}$.

In this section we treat the irreducible imprimitive rotation groups 
\[
	\Gs(kq,k,n) = \left\langle G(kq,k,n),\tau \right\rangle\subgr \SOr_{2n}
\]
where $n>2$, $q\in \N$, $k=1,2$, $kq\geq 3$ and where $\tau$ is a rotation that conjugates the first two coordinates, i.e. 
\[
	\tau(z_1,z_2,z_3\ldots,z_n)=(\overline{z}_1,\overline{z}_2,z_3\ldots,z_n),
\]
and the corresponding irreducible imprimitive reflection-rotation groups
\[
	\Gt(kq,k,n) = \left\langle G(kq,k,n),s \right\rangle\subgr \SOr_{2n},
\]
where $s$ is a reflection that conjugates the first coordinate, i.e. 
\[
	s(z_1,z_2,z_3\ldots,z_n)=(\overline{z}_1,z_2,\ldots,z_n)
\]
(cf. \cite{LaMik} for more details on the constructions of these groups). Let $\R^{2n} = V_1 + \dots + V_n$ be a decomposition into components of a system of imprimitivity of $\Gt(kq,k,n)$ (and hence of $\Gs(kq,k,n)$), i.e. a decomposition into subspaces that are permuted by the group. Let $H\triangleleft\Gt(kq,k,n)$ be the normal subgroup generated by rotations in $\Gt(kq,k,n)$ that only act in one of the factors $V_i$, $i=1,\ldots,n$. The projection $H_i$ of $H$ to $\Or(V_i)$ is a cyclic group of order $q$ and the group $H$ splits as a product of these projections, i.e. $H=H_1\times \dots \times H_n$. Because of $k\in \{1,2\}$ and $kq \geq 3$, the group $H$ is nontrivial. Due to Lemma \ref{lem:main_uptofour_sum} the PL linearization principle can be applied to the groups $H\triangleleft \Gt(kq,k,n)$. Hence, we have
\begin{lem}\label{lem:imprimi} Let $G$ be a reflection-rotation group of type $\Gt(kq,k,n)$ or $\Gs(kq,k,n)$ with $n>2$, $k=1,2$, $kq\geq 3$ and suppose that our main theorem holds for reflection-rotation groups of smaller order than $G$. Then our main result holds for the group $G$.
\end{lem}
Moreover, we see
\begin{lem}\label{lem:line_impr} Let $M\triangleleft G_{rr}$ be a pair occurring in Theorem \ref{thm:reducible_pairs} of type $\Gs(kq,k,n)\triangleleft \Gt(kq,k,n)$ or $\Gs(2q,2,n)\triangleleft \Gt(2q,1,n)$, $n>2$, $k=1,2$, $kq \geq 3$. Then there exists a nontrivial rotation group $H<M$ normalized by $G_{rr}$ such that the PL linearization principle can be applied to the groups $H \triangleleft G_{rr}$.
\end{lem}

Observe that by now we have verified the conclusion of the preceding lemma for all pairs of groups $M\triangleleft G_{rr}$ occurring in Theorem \ref{thm:reducible_pairs}, and hence established the respective condition in Lemma \ref{lem:redu_red} on reducible reflection-rotation groups.

\subsection{Exceptional rotation groups}
\label{sub:exceptional_groups_hom}

The only indecomposable rotation groups which we have not treated in view of our main result yet are the exceptional irreducible rotation groups $M(\Rr_5)$, $M(\Ss_6)$, $M(\Qq_7)$, $M(\Qq_8)$ and $M(\Tt_8)$ and the exceptional reducible rotation groups $\Delta_{\varphi}(W\times W)$ for $W$ of type $\Hn_3$ and $\Hn_4$ (cf. Section \ref{sub:red_pse_gro} and \cite[Sect.~4.6]{LaMik}). The proofs in \cite{Mikhailova} in the cases of $M(\Rr_5)$, $M(\Ss_6)$, $M(\Qq_7)$, $M(\Qq_8)$ in principle work \cite[II)-IV) in Thm.~1.4, p.~105]{Mikhailova} but lack some arguments. This manifests in the fact that isotropy groups, which determine the local structure of the respective quotient, are not examined. The cases of $M(\Tt_8)$ and $\Delta_{\varphi}(W\times W)$ for $W$ of type $\Hn_3$ and $\Hn_4$ are not considered in \cite{Mikhailova}.

In the cases $n>5$ we will make use of the PL version of the generalized Poincar\'e conjecture (cf. Theorem \ref{thm:poincare_pl_manifold}). For $n=5$ this tool is not available, which is why we have to perform a computation by hand in the case of $M(\Rr_5)$. Our arguments turn the approach in \cite[p.~102]{Mikhailova} to this case into a rigorous proof.

\begin{lem}\label{lem:R_5}
Our main result holds for the rotation group $M(\Rr_5)$, i.e. for $G=M(\Rr_5)<\SOr_5$ the PL quotient $\R^5/G$ is PL homeomorphic to $\R^5$.
\end{lem}
\begin{proof}
The outline of the proof is as follows. First we construct a fundamental domain $\Lambda$ on $S^4$ and choose an admissible triangulation that refines the tesselation of $S^4$ by the translates of $\Lambda$. With respect to the induced PL structure $\Lambda$ is a PL $4$-ball. We can choose a PL collar $A$ of $\partial \Lambda$ in $\Lambda$ that collapses onto $\partial \Lambda$ \cite[Cor.~3.17, Cor.~3.30]{MR0350744} (cf. Section \ref{sub:collapse}). The closure of the complement of $A$ in $\Lambda$, say $B$, is a PL $4$-ball. We set $Q=\partial \Lambda / \sim$ and $N=A / \sim$ where $\sim$ denotes the equivalence relation induced by $G$. Then we can recover $S^4/G$ by gluing together $N$ and $B$ along the PL $3$-sphere $\partial B$ which is not affected by $\sim$. In particular, we see that $S^4/G$ is a PL $4$-sphere, if we can show that $N$ is a PL $4$-ball. Due to Section \ref{sub:dim_four_groups_hom} and the fact that all isotropy groups of $G$ are rotation groups \cite[Lem.~27]{LaMik}, we already know that $N$ is a PL $4$-manifold with boundary (cf. Lemma \ref{lem:PL_isotropy_condition}). Moreover, since our triangulation is admissible, the projection $A \To N$ is simplicial and maps simplices homeomorphically onto simplices (cf. Section \ref{sub:ad_triangulations}). Hence, $N$ collapses onto $Q$ by Lemma \ref{lem:induced_collapse}. Therefore, according to Lemma \ref{lem:collapse} it is sufficient to show that $Q$ is collapsible in order to prove the lemma.

The group $G$ is isomorphic to the alternating group $\alt_5$ and a specific set of generators in $\alt_6$ is given by $(12)(34),(15)(23),(16)(24)$ (cf. \cite[p.~102]{Mikhailova}) where $G$ is regarded as the restriction of the permutation action of $\sym_6$ on $\R^6$ to the subspace $\R^5 = \{(x_1,\ldots,x_6)\in \R^6 | x_1+\ldots+x_6 =0\}$ of $\R^6$. A fundamental domain $\Lambda$ for the action of $G$ on $S^4$ is constructed in \cite[p.~103]{Mikhailova} as follows: For $v_0=\{-1,-1,-1,0,1,2\}$ we have $gv_0 \neq v_0$ for all $g\in G$ and thus
\[
		\Lambda = \bigcap_{g\in G} \{v\in S^4 | (v,v_0)\geq (v,gv_0)\}
\]
is a fundamental domain for the action of $G$ on $S^4\subseteq \R^5 \subseteq \R^6$. It can be described by the $8$ inequalities
\begin{align*}
	-x_3+x_4 & \geq 0, & -x_2+x_6 & \geq 0,\\
	-x_4+x_5 & \geq 0, & -x_1-2x_2+x_4+2x_5 & \geq 0,\\
	-x_5+x_6 & \geq 0, & -2x_1-x_2+x_4+2x_5 & \geq 0,\\
	-x_1+x_5 & \geq 0, & -2x_2-x_3+x_4+2x_5 & \geq 0 					
\end{align*}
and has vertices
\begin{align*}
	v_1 & = \frac{1}{\sqrt{30}}(-5,1,1,1,1,1), & v_2&= \frac{1}{\sqrt{30}}(1,-5,1,1,1,1),\\
	v_3 & = \frac{1}{\sqrt{30}}(1,1,-5,1,1,1), & v_4&= \frac{1}{\sqrt{30}}(-1,-1,-1,-1,-1,5),\\
	v_5 & = \frac{1}{\sqrt{6}}(1,-1,-1,-1,1,1), & v_6&= \frac{1}{\sqrt{84}}(1,1,-5,-5,4,4),\\
	v_7 & = \frac{1}{\sqrt{6}}(-1,1,-1,-1,1,1), & v_8&= \frac{1}{\sqrt{84}}(-5,4,-5,1,1,4). 				
\end{align*}
Let $P_i$ be the boundary of the half-space in $\R^6$ determined by the $i$th inequality above. The faces of the fundamental domain are the following three-dimensional polytopes: In the plane $P_1$ the pentagonal pyramid $v_1v_2v_4v_5v_6v_7$ with vertex $v_4$; in $P_2$ the double pyramid $v_1v_2v_3v_4v_8$ with vertices $v_2$ and $v_8$; in $P_3$ the pentagonal pyramid $v_1v_2v_3v_5v_6v_7$ with vertex $v_3$; in $P_4$ the simplex $v_2v_3v_4v_5$; in $P_5$ the simplex $v_1v_3v_7v_8$; in $P_6$ the double pyramid $v_3v_4v_6v_7v_8$ with vertices $v_6$ and $v_8$; in $P_7$ the simplex $v_3v_4v_5v_6$ and in $P_8$ the simplex $v_1v_4v_7v_8$. The boundary of the fundamental domain is illustrated in Figure \ref{fig:Boundary_fundamental_domain}.

It can be checked that the identifications induced by $G$ on $\partial \Lambda$ are generated by the following identifications

\begin{align*} 
(12)(34) \text{ in }P_1: 	&v_1 \rightleftharpoons v_2,  v_4 \rightleftharpoons v_4,v_6 \rightleftharpoons v_6, v_5 \rightleftharpoons v_7,\\
(13)(45)\text{ in }P_2: 	& v_1 \rightleftharpoons v_3, v_2 \rightleftharpoons v_2, v_4 \rightleftharpoons v_4, v_8 \rightleftharpoons v_8,\\
(12)(56)\text{ in }P_3: 	&v_1 \rightleftharpoons v_2, v_3 \rightleftharpoons v_3,v_5 \rightleftharpoons v_7, v_6 \rightleftharpoons v_6,\\
(15)(23)\text{ in }P_4: 	&v_2 \rightleftharpoons v_3, v_4 \rightleftharpoons v_4, v_5 \rightleftharpoons v_5,\\
(13)(26)\text{ in }P_5: 	&v_1 \rightleftharpoons v_3, v_7 \rightleftharpoons v_7, v_8 \rightleftharpoons v_8,\\
(14)(25)\text{ in }P_6: 	&v_3 \rightleftharpoons v_3, v_4 \rightleftharpoons v_4, v_6 \rightleftharpoons v_8, v_7 \rightleftharpoons v_7,\\ 
(13425):	&v_1 \rightharpoonup v_3, v_4 \rightleftharpoons v_4, v_7 \rightharpoonup v_5, v_8 \rightharpoonup v_6,\\
(15243):	&v_1 \leftharpoondown v_3, v_4 \rightleftharpoons v_4, v_7 \leftharpoondown v_5, v_8 \leftharpoondown v_6,
\end{align*}

The first six of them correspond to ``pasting in half'' the faces of $\Lambda$ lying in the planes $P_1,\ldots ,P_6$. Since points in the interior of such a face are not identified with points outside the interior of this face, we see that the images of these faces in the quotient $Q$ can be collapsed to the images of their boundary in $Q$. In particular, we see that $Q$ collapses onto the images of the faces $v_3v_4v_5v_6$ and $v_1v_4v_7v_8$ of $\Lambda$. Since these faces are identified with each other by $(13425)$, we see that $Q$ collapses onto the image of $v_3v_4v_5v_6$ in $Q$. Examining the list of generating identifications shows that this image is a $3$-simplex itself. Hence, $Q$ is collapsible and thus the claim follows by the remarks above.

\begin{figure}
	\centering
		\includegraphics[width=0.9\textwidth]{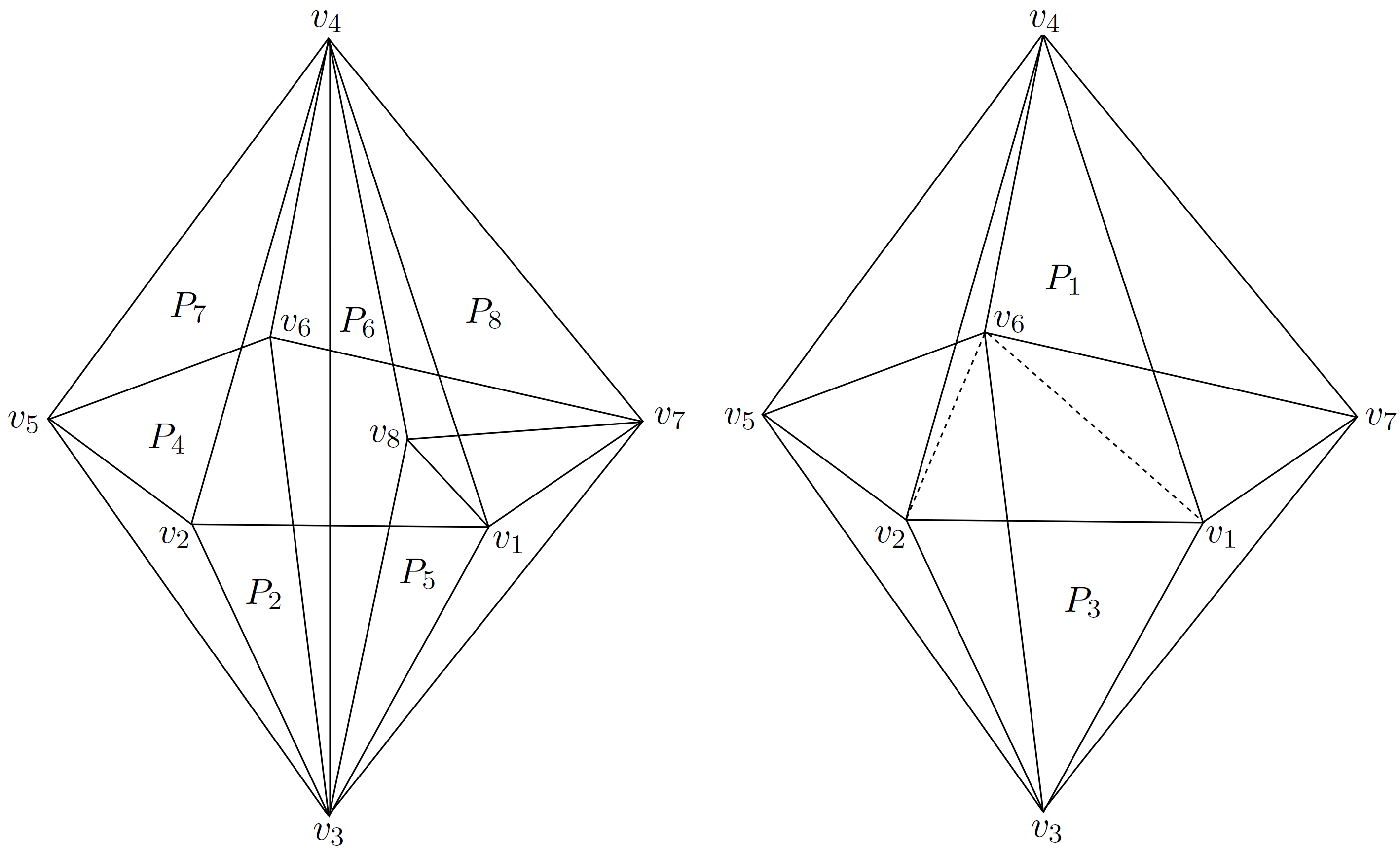}
	\caption{Boundary of a fundamental domain $\Lambda$ for the action of $M(\Ss_5)=R_5(\alt_5)$ on $S^4$ cut into two pieces.}
	\label{fig:Boundary_fundamental_domain}
\end{figure}

\end{proof}

In principle, the proofs in \cite{Mikhailova} in the cases $M(\Ss_6)$, $M(\Qq_7)$ and $M(\Qq_8)$ can be made rigorous in the same way. However, in order to avoid long computations with fundamental domains and identifications, we provide the following alternative argument.

In each of the remaining cases choose an admissiable triangulation of $S^{n-1}$ for the action of $G$. Because of $n>5$ in these cases, it is sufficient to show that $S^{n-1}/G$ is a simply connected PL manifold with $H_*(S^{n-1}/G)=H_*(S^{n-1})$ in order to prove that $\R^n/G$ is piecewise linear homeomorphic to $\R^n$ (cf. Section \ref{sub:poincare}, Theorem \ref{thm:poincare_pl_manifold}). According to \cite[Lem.~27]{LaMik} all isotropy groups of the remaining irreducible exceptional rotation groups $M(\Ss_6)$, $M(\Qq_7)$, $M(\Qq_8)$ and $M(\Tt_8)$ are rotation groups. The same statement is true for reducible rotation groups of type $\Delta_{\varphi}(W\times W)$ (cf. Section \ref{sub:red_pse_gro}), since isotropy groups of real reflection groups are generated by the reflections they contain \cite[Thm. 1.12 (c), p.~22]{MR1066460} (apply this result twice). Hence, we have (recall that the quotient $S^{n-1}/G$ is simply connected by Lemma \ref{lem:simply_con_quotient})

\begin{lem}
Let $G$ be one of the remaining rotation groups and suppose that our main result holds for all rotation groups of smaller order. Then $S^{n-1}/G$ is a simply connected PL manifold.
\end{lem}

In particular, we have $H_1(S^{n-1}/G)=\pi_1(S^{n-1}/G)_{ab}=0$ \cite[Thm.~2A.1]{MR1867354}. By Poincar\'e duality and the universal coefficient theorem, it is sufficient to show that $H_2(S^{n-1}/G)=0$ or $H_2(S^{n-1}/G)=H_3(S^{n-1}/G)=0$, respectively, depending on whether $n\leq 6$ or $6<n\leq 8$ \cite[Thm.~3.2, Thm.~3.30]{MR1867354}, in order to prove  $H_*(S^{n-1}/G)=H_*(S^{n-1})$. In Lemma \ref{lem:homology_condition} we have seen that the existence of subgroups $H<G$ with coprime indices and $H_i(S^{n-1}/H)=0$ implies $H_i(S^{n-1}/G)=0$. The latter condition in particular holds for rotation groups $H$ for which we have already shown that $\R^n/H$ is homeomorphic to $\R^n$ (cf. the long exact sequence in \cite[p.~117]{MR1867354}). Hence, in order to treat the remaining cases, it is sufficient to find suitable subgroups.

\begin{lem}\label{lem:S_6}
If our main result holds for all rotation groups of smaller order than  $M(\Ss_6)$, then it also holds for  $M(\Ss_6)$.
\end{lem}
\begin{proof} The group $M(\Ss_6)\cong \mathrm{PSL}_2(7)$ has order $2^{3}\cdot3\cdot 7$ and thus contains a subgroup $H$ of order $7$. Since $M(\Ss_6)$ can be realized as the isomorphic image of a permutation group in $S_7<\SOr_7$ to $\SOr_6$ (cf. \cite[Lem.~16]{LaMik}), this subgroup of order $7$ is generated by a $7$-cycle in $S_7$ and acts thus freely on the unit sphere $S^5\subset\R^6$. The facts that $H_2(\Z_7)=0$ \cite[(3.1), p.~35]{MR1324339} and that $H_2(S^5/H)=H_2(H)$ for groups acting freely on $S^5$ \cite[p.~20]{MR1324339} imply that $H_2(S^5/H)=0$. A rotation group of order $2^{3}\cdot 3$ contained in $M(\Ss_6)$ is described in \cite[Lem.~28]{LaMik}. Hence, the claim follows by the remarks above.
\end{proof}

The reflection groups $W(\Hn_3)$ and $W(\Hn_4)$ properly contain reflection groups with coprime indices (cf. \cite[Table 8, Table 9]{MR3169410}) and thus the rotation groups $\Delta_{\varphi}(W\times W)$ for $W$ of type $\Hn_3$ and $\Hn_4$ properly contain rotation groups with coprime indices. The rotation groups $M(\Qq_7)$ and $M(\Tt_8)$ also properly contain rotation groups with coprime indices \cite[Lem.~29, Lem.~31]{LaMik}. By the remarks above, we obtain
\begin{lem} \label{lem:nurnochQ7}
Let $G$ be a rotation group of type $M(\Qq_7)$, $M(\Tt_8)$ or $\Delta_{\varphi}(W\times W)$ for $W$ of type $\Hn_3$ or $\Hn_4$. If our main result holds for all rotation groups of smaller orden than $G$, then it also holds for $G$.
\end{lem}
The group $M(\Qq_7)$ does not contain rotation groups with coprime indices. However, the following lemma is proven in \cite[Lem.~30]{LaMik}.
\begin{lem}
The rotation group $M(\Qq_8)$ of order $21504=2^{10}\cdot3\cdot 7$ contains a reducible rotation group $G$ of order $1536=2^{9}\cdot3$ with $k=2$ and
\[
	(G_i,H_i,F_i,G_i/H_i)=(W(\D_4),D(\Wp(\D_4)),W(\D_4),W(\A_3)),
\]
$i=1,2$ (see below for an explanation of this notation), which is normalized by an element $\tau$ of order two that interchanges the two components of $G$. Moreover, it contains a rotation group $M(\Ss_6)$ of order $168=2^{3}\cdot3\cdot 7$.
\end{lem}
More precisely, with respect to the action of the group $G<\SOr_8$ in the preceding lemma $\R^8$ splits into two four-dimensional irreducible subspaces $\R^8=V_1\oplus V_2$. The subgroup $H_i$ of $G$ generated by rotations that only act in $V_i$ is of type $D(\Wp(\D_4))$. The quotient of $G$ by $H=H_1 \times H_2$ is abstractly isomorphic to $W(\A_3)\cong \sym_4$ and its linearization on $\R^8$ acts as a diagonal group $\Delta_{\varphi}(W\times W)$ for $W$ of type $\A_3$ (cf. Section \ref{sub:red_pse_gro}).
\begin{lem}
For $N=\left\langle G, \tau \right\rangle$ as in the preceding lemma we have $\R^8/N\cong C(S^3 * \R\mathbb{P}^3)$ and hence $H_*(S^7/N)=H_*(\Sigma^4(\R\mathbb{P}^3))$. In particular, $H_2(S^7/N)=H_3(S^7/N)=0$.
\end{lem}
\begin{proof} Let $\R^8=V_1+V_2$ be a decomposition into irreducible components with respect to $G$. We can assume that $V_1$ and $V_2$ are two identical copies of $\R^4$ such that the projections of $G$ to $V_1$ and $V_2$ coincide. Since $\tau$ normalizes $G$ it has the form
\[
	\begin{array}{ccl}
		\tau  = \left(
		  \begin{array}{cc}
		    0 & h \\
		    h^{-1} & 0 \\
		  \end{array}
		\right)
  \end{array}
\]
for some $h\in N_{\Or_4}(W(\D_4))$. After conjugation we can assume that $h=\mathrm{id}$ and thus we can apply the linearization principle to the groups $H_1\times H_2\triangleleft N$. It then remains to show that $\R^8/\overline{N}\cong C(S^3 * \R\mathbb{P}^3)$ for $\overline{N}=\left\langle \overline{G}, \sigma \right\rangle$ with $\sigma=\overline{\tau}$ and $\overline{G}=G_{\varphi}=\{(g,\varphi(g))| g \in W(\A_3)\}\subset  W(\A_3)\times W(\A_3) <\Or_4\times \Or_4$ for some $\varphi\in \mathrm{Aut}(W(\A_3))$. Since the symmetric group $\sym_4$ has no outer automorphisms we can assume that $\varphi=\mathrm{id}$ after conjugation. Then $\sigma$ has the form
\[
	\begin{array}{ccl}
		\sigma  = \left(
		  \begin{array}{cc}
		    0 & g_0 \\
		    g_0^{-1} & 0 \\
		  \end{array}
		\right)<\SOr_8
  \end{array}
\]
for some $g_0 \in W(\A_3) < \Or_4$. Now $\sigma \in N(\overline{G})$ implies $g_0^2 \in C_{\Or_4}(W(\A_3))=\{\pm \mathrm{id}_3 \}\times \{ \pm \mathrm{id}_1\} < \Or_3 \times \Or_1$. Because of $g_0 \in W(\A_3)$ and since $W(\A_3)$ has a trivial center, we have $g_0^2= 1$, i.e. $g_0=g_0^{-1}$. If we identify $\R^8$ with $\Co^4$ then $\overline{G}$ can be regarded as a unitary reflection group and the action of $\sigma$ is given by $\sigma\left((z_1,z_2,z_3,z_4)\right)=i\cdot g_0 (\overline{z}_1,\overline{z}_2,\overline{z}_3,\overline{z}_4)$ where $g_0$ permutes the coordinates. Let $s_i$ be the elementary symmetric polynomial of degree $i$ in the $z_j$, $i,j=1,2,3,4$. Then the map 
\[
		 \begin{array}{cccl}
		 f : & \Co^4/\overline{G} 				& \To  	&	\Co^4 \\
		           		& (z_1,z_2,z_3,z_4) 	& \MTo  & (s_1,s_2,s_3,s_4)
		 \end{array}
\]
defines a homeomorphism as explained in Section \ref{sub:alg_inv}. Since the $s_i$ are invariant under coordinate permutations, the induced action of $\overline{N}/\overline{G}$ on $\Co^4$ is given by 
\[
		\overline{\sigma}(s_1,s_2,s_3,s_4)=(i\cdot\overline{s}_1,-\overline{s}_2,-i\cdot\overline{s}_3,-\overline{s}_4).
\]
It follows that $\R^8/N\cong C(S^3 * \R\mathbb{P}^3)$. In particular, we have $H_2(S^7/N)=H_3(S^7/N)=0$ by the long exact sequence in \cite[p.~117]{MR1867354}.
\end{proof}
With the preceding two lemmas and the remarks above we obtain
\begin{lem} \label{lemQ_8}
If our main result holds for all rotation groups of smaller order than $M(\Qq_8)$, then it also holds for $M(\Qq_8)$.
\end{lem}
Since we have by now treated all cases, our main result follows by induction.

\section{Towards a classification free proof}
\label{ch:outlook}
As a corollary of our main result reflection-rotation groups share the following properties $(i)-(iii)$ (cf. Lemma \ref{lem:PL_isotropy_condition}). Property $(i)$ generalizes an isotropy theorem by Steinberg on unitary reflection groups \cite[Thm. 1.5, p.~394]{MR0167535} (cf. \cite[Ch. V, Exercise 8]{MR1080964} and \cite{MR2052515} for alternative proofs by Bourbaki and Lehrer).
\begin{enumerate}
\item Isotropy groups of reflection-rotation groups are generated by the reflections and rotations they contain.
\item For a rotation group $G<\SOr_n$ we have\[H_*(S^{n-1}/G;\Z)=H_*(S^{n-1};\Z).\]
\item For a reflection-rotation group $G<\Or_n$ that contains a reflection we have \[H_*(S^{n-1}/G;\Z)=H_*(\{*\};\Z).\]
\end{enumerate}
Conversely, these properties together with the PL h-cobordism theorem and some extra work in low dimensions imply the if direction of our result by induction as explained in Section \ref{sub:poincare}. Hence, in order to essentially dispense with the classification of reflection-rotation groups in the proof of our result one might first try to find conceptual proofs for properties $(i)$, $(ii)$ and $(iii)$ that do not depend on a classification of reflection-rotation groups.

\setcounter{secnumdepth}{-2}
\renewcommand\thesection{\Alph{section}}

\setcounter{secnumdepth}{1}
\setcounter{section}{0}
\section{Appendix}

\subsection{Classification of rotation-reflection groups}
\label{app:class-rr-groups}

We provide the classification result from \cite{LaMik} omitting details that are not needed in the present paper. For more detailed descriptions we refer to the respective sections above and to \cite{LaMik}.

\begin{thm}\label{thm:clas_red_psg}
Every irreducible rotation group occurs, up to conjugation, in precisely one of the following cases.
\begin{compactenum}
\item Orientation preserving subgroups $\Wp$ of irreducible real reflection groups $W$ (Lem. \ref{lem:main_refl}).
\item Irreducible unitary reflection groups that are not the complexification of a real reflection group considered as real groups (Lem. \ref{lem:main_unit}).
\item The imprimitive rotation groups $\Gs(km,k,l)\subgr\SOr_{n}$ for $n=2l>4$, $k\in\{1,2\}$ and $km\geq3$ (Lem. \ref{lem:imprimi}).
\item The unique extensions $\Ws$ of $\Wp$ by a normalizing rotation if such exists (Lem. \ref{lem:main_refl}).
\item The following rotation groups $M$ of type $\Pp_5$, $\Pp_6$, $\Pp_7$, $\Pp_8$, $\Qq_7$, $\Qq_8$, $\Rr_5$, $\Ss_6$ and $\Tt_8$. More precisely (in the same order),
\begin{compactenum}
\item monomial rotation groups $M_5$, $M_6$, $M_7$, $M_8$ (Lem. \ref{lem:monom}) and $M^p_7$ (Lem. \ref{lem:nurnochQ7}) and $M^p_8$ (Lem. \ref{lemQ_8}).
\item primitive rotation groups $R_5(\alt_5)$ (Lem. \ref{lem:R_5}) and $R_6(\PSL_2(7))$ (Lem. \ref{lem:S_6}) given as the image of the unique irreducible representations of $\alt_5$ in $\SOr_5$ and of $\PSL_2(7)$ in $\SOr_6$,
\item a primitive rotation group in $\SOr_8$ isomorphic to an extension of the alternating group $\alt_8$ by a nonabelian group of order $2^8$ (Lem. \ref{lem:nurnochQ7}).
\end{compactenum}
\item Other examples in dimension $4$ (Lem. \ref{lem:main_uptofour_sum}).
\end{compactenum}
The references in the parentheses indicate where the respective group is treated in view of our main result.
\end{thm}

For a real reflection group $W$ we denote by $\Wt$ its unique extension by a normalizing rotation, provided such exists, i.e. $\Wt=\left\langle \Ws,W \right\rangle$. For a monomial rotation group $M$ we denote by $\Mt$ its extension by a coordinate reflection (cf. Section \ref{sub:monomial_groups_hom}). Finally, for an imprimitive rotation group of type $G(km,k,l)$ let $\Gt (km,k,l)$ be its extension by a reflection $s$ of the form $s(z_1,\ldots, z_l)=(\overline{z}_1,z_2,\ldots, z_l)$ (cf. Section \ref{sub:imprimitive_rot_groups}). The following result holds (cf. \cite[Thm.~2]{LaMik}).
\begin{thm}\label{thm:irr_re-ro_group}
Every irreducible reflection-rotation group either appears in Theorem \ref{thm:clas_red_psg} or it contains a reflection and occurs, up to conjugation, in one of the following cases.
\begin{compactenum}
\item Irreducible real reflection groups $W$ (Lem. \ref{lem:main_refl}).
\item The groups $\Wt$ generated by a reflection group $W$ of type $\A_4$, $\D_4$, $\F_4$, $\A_5$ or $\E_6$ and a normalizing rotation (Lem. \ref{lem:main_refl}).
\item The monomial groups $\Mt$ of type $\D_n$, $\Pp_5$, $\Pp_6$, $\Pp_7$, $\Pp_8$, i.e. $\Mt(\D_n):=D(W(\BC_n)) \rtimes \alt_n$, $\Mt_5$, $\Mt_6$, $\Mt_7$ and $\Mt_8$ (Lem. \ref{lem:monom}).
\item The imprimitive groups $\Gt (km,k,l)\subgr\SOr_{n}$ with $n=2l$, $k=1,2$ and $km\geq3$ (Lem. \ref{lem:imprimi}).
\end{compactenum}
The references in the parentheses indicate where the respective group is treated in view of our main result.
\end{thm}

For a reducible reflection-rotation group $G\subgr \Or_n$ let $\R^n=V_1 \oplus \ldots \oplus V_k$ be a decomposition into irreducible components, let $G_i$ be the projection of $G$ to $O(V_i)$ and let $H_i\triangleleft G_i$ be the normal subgroup of $G$ generated by rotations in $G$ that only act in $V_i$. Then the group $G_i$ is generated by its reflections and by $H_i$. If $G_i$ does not split off as a direct factor, then $G_i$ contains a reflection and is thus distinct from $H_i$. Conversely, every pair $M \triangleleft G_{rr}$ of an irreducible reflection-rotation group $G_{rr}$ that contains a reflection and a normal subgroup $M \triangleleft G_{rr}$ generated by rotations such that $G_{rr}$ is generated by its reflections and by $M$ occurs in this way. The following theorem is proven in \cite[Thm.~3]{LaMik} (cf. the remark following \cite[Thm.~3]{LaMik}).

\begin{thm} \label{thm:reducible_pairs} Let $G_{rr}<\Or_n$ be an irreducible reflection-rotation group that contains a reflection and let $M \triangleleft G_{rr}$ be a nontrivial normal subgroup generated by rotations. If $G_{rr}$ is generated by the reflections it contains and by $M$, then the pair $M \triangleleft G_{rr}$ occurs in precisely one of the following cases.
\begin{enumerate}
\item $M \triangleleft \Mt$ for $M=M_5, M_6,M_7, M_8, M(\D_n)=\Wp(\D_n)$ \hfill (Lem. \ref{lem:monom})
\item $\Gs(km,k,l)\triangleleft \Gt (km,k,l)$ for $k=1,2$, $km\geq 3$ and $n=2l$ \hfill (Lem. \ref{lem:line_impr})
\item $\Gs(2m,2,l)\triangleleft \Gt (2m,1,l)$ for $m\geq 2$ and $n=2l$ \hfill (Lem. \ref{lem:line_impr}) 
\item $\Wp \triangleleft W$ for any irreducible reflection group $W$. \hfill (Lem. \ref{lem:main_refl})
\item $D(\Wp(\BC_n))\triangleleft W(\BC_n)$ \hfill (Lem. \ref{lem:line_refl})
\item $\Wp(\D_n) \triangleleft W(\BC_n)$ \hfill (Lem. \ref{lem:main_refl})
\item $D(W(\D_n)) \triangleleft W(\D_n)$ \hfill (Lem. \ref{lem:line_refl})
\item $\Ws \triangleleft \Wt$ for a reflection group $W$ of type $\A_5$ or $\E_6$. \hfill (Lem. \ref{lem:main_refl})
\item Other examples in dimensions $n\leq 4$ \hfill (Lem. \ref{lem:main_uptofour_sum})
\end{enumerate}
In the lemmas referred to, we either show the existence of a nontrivial rotation group $H\triangleleft M$ normalized by $G_{rr}$ for which the linearization principle can be applied to the groups $H\triangleleft G_{rr}$ directly or reduce such a claim to the validity of our main result for reflection-rotation groups of order less than $G_{rr}$.
\end{thm}

\end{document}